\documentclass[a4paper, 12pt]{article}

\usepackage{amsfonts, amsmath, amsthm}
\usepackage{latexsym}

\theoremstyle{plain}
\newtheorem{thm}{Theorem}[section]
\newtheorem{prp}{Proposition}[section]
\newtheorem{lem}{Lemma}[section]

\theoremstyle{definition}

\theoremstyle{remark}
\newtheorem{rmk}{Remark}[section]

\numberwithin{equation}{section}

\newcommand{\R}{\mathbb{R}}
\newcommand{\C}{\mathbb{C}}

\newcommand{\pa}{\partial}
\newcommand{\eps}{\varepsilon}
\newcommand{\jb}[1]{\langle #1 \rangle}

\newcommand{\JB}[1]{\Bigl\langle #1 \Bigr\rangle}

\DeclareMathOperator{\realpart}{\rm Re}
\DeclareMathOperator{\imagpart}{\rm Im}

\newcommand{\dis}{\displaystyle}

\begin{document}
\title{
 A remark on decay rates of solutions for a system of quadratic 
nonlinear Schr\"odinger equations in 2D
 }  

\author{
          Soichiro Katayama\thanks{
              Department of Mathematics, Wakayama University. 
              930 Sakaedani, Wakayama 640-8510, Japan. 
              (E-mail: {\tt katayama@center.wakayama-u.ac.jp})}
         \and Chunhua Li\thanks{
              Department of Mathematics, Yanbian University. 
              Yanji 133002, China. 
              (E-mail: {\tt sxlch@ybu.edu.cn})}
         \and  Hideaki Sunagawa\thanks{
              Department of Mathematics, Graduate School of Science, 
              Osaka University. 
              1-1 Machikaneyama-cho, Toyonaka, Osaka 560-0043, Japan. 
              (E-mail: {\tt sunagawa@math.sci.osaka-u.ac.jp})}
} 
 
\date{\today }   
\maketitle

\noindent{\bf Abstract:}\ 
We consider the initial value problem for a three-component system of 
quadratic nonlinear Schr\"odinger equations with mass resonance 
in two space dimensions. 
Under a suitable condition on the coefficients of the nonlinearity, 
we will show that the solution decays strictly faster than ${O(t^{-1})}$ 
as $t \to +\infty$ in $L^{\infty}$ by providing with an enhanced 
decay estimate {of order $O((t\log t)^{-1})$}. 
Differently from the previous works, our approach does not rely on the 
explicit form of the asymptotic profile of the solution at all. \\

\\

\\

\noindent{\bf Running Title:}\ 
Decay for a system of NLS
\\

\section{Introduction and the main result}
This paper is intended to be a sequel of the papers \cite{L1} and 
\cite{L2} by one of the authors, which are concerned with decay property of 
solutions to the initial value problem for a class of nonlinear Schr\"odinger 
systems. 
The model system which we focus on here is 
\begin{align}
 \left\{\begin{array}{l}
 \dis{i\pa_t u_1 + \frac{1}{2m_1}\Delta u_1 
      = \lambda_1 |u_1|u_1 +\mu_1 \overline{u_2} u_3, }\\
 \dis{i\pa_t u_2 + \frac{1}{2m_2}\Delta u_2 
      = \lambda_2 |u_2|u_2 +\mu_2 \overline{u_1} u_3, }\\
 \dis{i\pa_t u_3 + \frac{1}{2m_3}\Delta u_3 
      = \lambda_3 |u_3|u_3 +\mu_3  u_1 u_2, }
 \end{array}\right.
 \quad t>0,\ \ x \in \R^2
 \label{eq}
\end{align}
with 
\begin{align}
 u_j(0,x)=\varphi_j(x), \qquad x \in \R^2, \ \ j=1,2,3
 \label{data}
\end{align}
(see, e.g., \cite{CC}, \cite{CCO}, \cite{HLN1} for the physical background of 
this system), where $m_1, m_2, m_3 \in \R \setminus \{0\}$, 
$\lambda_1, \lambda_2, \lambda_3$, 
$\mu_1, \mu_2, \mu_3 \in \C\setminus \{0\}$ are constants, 
$\varphi=(\varphi_j(x))_{j=1,2,3}$ is a prescribed $\C^3$-valued function, and 
$u=(u_j(t,x))_{j=1,2,3}$ is a $\C^3$-valued unknown function.
As usual, $i=\sqrt{-1}$, $\pa_t=\pa/\pa t$, and $\Delta$ is the Laplacian in $x$-variables.

By a minor modification of the method of \cite{L1} and \cite{L2},  
we can show the following basic $L^{\infty}$-decay result. 
Here and hereafter, we denote by 
$H^{s,\sigma}(\R^2)$ the weighted Sobolev spaces, i.e.,  
$$
 H^{s,\sigma}(\R^2)
 =
 \{ f \in L^2(\R^2)\, |\, 
    (1+|x|^2)^{\sigma/2}(1-\Delta)^{s/2}f \in L^2(\R^2)
 \}
$$
equipped with the norm 
$\|f\|_{H^{s,\sigma}} 
=\left\|(1+|x|^2)^{\sigma/2}(1-\Delta)^{s/2}f\right\|_{L^{2}}$.

\begin{prp}\label{prp_basic}
Assume 
\begin{align}
  m_1+m_2=m_3,
  \label{assump_m}
\end{align} 
\begin{align}
  \imagpart \lambda_j\le 0 \ \ \mbox{for $j=1,2,3$},
  \label{assump_lambda}
\end{align} 
\begin{align}
  \kappa_1\mu_1+\kappa_2\mu_2=\kappa_3 \overline{\mu_3} 
  \ \ \mbox{with some $\kappa_1$, $\kappa_2$, $\kappa_3>0$},
  \label{assump_mu}
\end{align} 
and 
$\varphi=(\varphi_j)_{j=1, 2, 3}\in H^{s,0}(\R^2)\cap H^{0,s}(\R^2)$ with $1<s < 2$.  
Then there exists a positive constant $\eps_0>0$ such that the initial 
value problem \eqref{eq}--\eqref{data} admits a unique global solution  
$$ 
 u  \in C\bigl([0,\infty); H^{s,0}(\R^2)\cap H^{0,s}(\R^2)\bigr), 
$$ 
provided that $\|\varphi\|_{H^{s,0}}+\|\varphi\|_{H^{0,s}}\le \eps_0$. 
Moreover, there exists a positive constant $C_0$ such that 
\begin{align}
 \|u(t,\cdot)\|_{L^{\infty}} \le \frac{C_0}{1+t}
  \label{decay1}
\end{align} 
for all $t\ge 0$.
\end{prp}

In fact, $\lambda_j|u_j|u_j$ was not included in the nonlinearity 
considered in \cite{L1} and \cite{L2}. 
However, as verified in Section~\ref{sec_apriori}, 
it is straightforward to modify 
the proof under the assumption \eqref{assump_lambda} on the coefficient 
$\lambda_j$. 
In view of the conservation law
$$
\sum_{j=1}^{3}\kappa_j 
 \left(
  \|u_j(t,\cdot)\|_{L^2}^2 
  - 2\imagpart \lambda_j \int_0^t\|u_j(\tau,\cdot)\|_{L^3}^3d\tau
 \right)
=
\sum_{j=1}^{3}\kappa_j \|u_j(0,\cdot)\|_{L^2}^2, 
$$
it may be natural to regard  \eqref{assump_lambda} as 
a kind of  dissipativeness condition. 
This will lead us to the following question:  
{\em 
Is the decay rate $O(t^{-1})$ in \eqref{decay1} enhanced
if the inequalities in \eqref{assump_lambda} are strict?} 
To the authors' knowledge, there is no previous result which answers this 
question except the case where the system is decoupled, i.e., 
$\mu_1=\mu_2=\mu_3=0$. In the decoupled case, the problem is reduced  to the  
single equation 
\begin{align}
 i\pa_t v+\frac{1}{2}\Delta v =\lambda |v|v,
  \quad t>0,\ \ x \in \R^2
\label{singleNLS}
\end{align}
with $\imagpart \lambda<0$, and the solution $v(t,x)$ decays like 
$O((t\log t)^{-1})$ in the sense of $L^{\infty}$ for small
initial data, according to \cite{Sh} 
(see also \cite{HN}, \cite{HNS} and \cite{Su}). 
However, the proof of \cite{Sh}  is not applicable to the present 
setting because it heavily depends on the facts that the solution $v(t,x)$ to 
\eqref{singleNLS} 
is well-approximated by 
$$
 \frac{e^{i\frac{|x|^2}{2t}}}{it} \alpha\left(t,\frac{x}{t}\right), 
$$
where $\alpha(t,\xi)$ solves 
$$
 i\pa_t \alpha  =\frac{\lambda}{t} |\alpha|\alpha  +o(t^{-1}),
$$ 
and that $\alpha(t,\xi)$ behaves like 
$$
  \frac{\alpha(1, \xi) }
  {1-|\alpha(1, \xi)|\imagpart \lambda \log t}
 \exp 
  \left(
    i\frac{\realpart \lambda}{\imagpart \lambda} 
    \log(1- |\alpha(1, \xi)|\imagpart \lambda \log t)
  \right)
$$
as $t\to +\infty$ if $\imagpart \lambda<0$. 
On the other hand, the corresponding reduced ODE system for \eqref{eq} is 
\begin{align}
 \left\{\begin{array}{l}
 \dis{i\pa_t \alpha_1
 =\frac{1}{t} \left( \lambda_1 |\alpha_1|\alpha_1 
  + 
  \mu_1 \overline{\alpha_2} \alpha_3 \right) +o(t^{-1}),}\\[3mm]
 \dis{i\pa_t \alpha_2 
 =\frac{1}{t} \left( \lambda_2 |\alpha_2|\alpha_2 
  + 
  \mu_2 \overline{\alpha_1} \alpha_3 \right) +o(t^{-1}),}\\[3mm]
 \dis{i\pa_t \alpha_3
 =\frac{1}{t} \left( \lambda_3 |\alpha_3|\alpha_3 
  + 
  \mu_3  \alpha_1 \alpha_2 \right) +o(t^{-1}),}
 \end{array} \right.
\label{profile}
\end{align}
which is much more complicated than the single case. 

The  aim of this  paper is  to give an affirmative answer to the above 
question by providing with an enhanced decay estimate 
{of order $O((t\log t)^{-1})$}. 
The novelty of the present approach is that it does not rely on 
the explicit form of the asymptotic profile of the solution at all. 
The main result is as follows.

\begin{thm} \label{thm_main}
Suppose that the assumptions of Proposition~\ref{prp_basic} are fulfilled.
Let $u$ be the  solution to \eqref{eq}--\eqref{data} whose existence is 
guaranteed by Proposition~\ref{prp_basic}. 
If 
\begin{align}
 \imagpart \lambda_j<0 \ \ \mbox{for $j=1,2,3$}
 \label{assump_lambda_strict}
\end{align} 
is satisfied, then 
there exists a positive constant $C_1$ such that 
\begin{align*}
 \|u(t,\cdot)\|_{L^{\infty}} 
 \le 
 \frac{C_1}{(1+t)\log (2+t)}
\end{align*} 
for all $t\ge 0$.
\end{thm}

\section{Preliminaries} \label{sec_aux}
In this section, we introduce several notations and lemmas which will be  
needed in the subsequent sections. 
In what follows, we denote by $\jb{\cdot, \cdot}_{\C^3}$ the standard scalar 
product in $\C^3$, i.e., 
$$
 \jb{z,w}_{\C^3}=\sum_{j=1}^{3}  z_j \overline{w_j}
$$
for $z=(z_j)_{j=1,2,3}$, $w=(w_j)_{j=1,2,3} \in \C^3$. 
We also write $|z|_{\C^3}=\sqrt{\jb{z,z}_{\C^3}}$, as usual. \\

First we rewrite \eqref{eq} in the abstract form: 
Put $\Lambda u=(\frac{1}{2m_j}\Delta u_j)_{j=1,2,3}$ so that  
$$
 i\pa_t u +\Lambda u=F(u),
$$
where 
$F: \C^3 \ni z=(z_j)_{j=1,2,3} \mapsto (F_j(z))_{j=1,2,3} \in \C^3$ is 
defined by
\begin{align*}
 &F_1(z)=\lambda_1|z_1|z_1 +\mu_1 \overline{z_2}z_3, \\ 
 &F_2(z)=\lambda_2|z_2|z_2 +\mu_2 \overline{z_1}z_3, \\ 
 &F_3(z)=\lambda_3|z_3|z_3 +\mu_3 z_1 z_2.
\end{align*}
Then the assumptions \eqref{assump_m}, \eqref{assump_lambda} and 
\eqref{assump_mu} on the coefficients can be interpreted as follows: 
\begin{itemize}
\item 
We define $\mathcal{E}_{\theta}:\C^3\to \C^3$ by 
$\mathcal{E}_{\theta} z = (e^{im_j \theta} z_j)_{j=1,2,3}$ for 
$z=(z_j)_{j=1,2,3} \in \C^3$ and $\theta\in \R$.
Then we have 
\begin{align}  
  \mathcal{E}_{\theta} F (z) = F(\mathcal{E}_{\theta} z) 
 \label{gauge_inv}
\end{align}
for all $\theta \in \R$ and $z \in \C^3$, 
provided that \eqref{assump_m} is satisfied.

\item
 With $\kappa_1$, $\kappa_2$, $\kappa_3$ appearing in \eqref{assump_mu}, 
we put 
$$
 A=
 \begin{pmatrix} 
  \kappa_1& 0 & 0\\ 0 &\kappa_2 & 0\\ 0 &0  & \kappa_3
 \end{pmatrix}.
$$
Then, by \eqref{assump_mu}, we have
\begin{align*}
 \jb{F(z), Az}_{\C^3} 
 = \sum_{j=1}^{3} \kappa_j \lambda_j |z_j|^3 
   + 2\kappa_3 \realpart\Bigl(\overline{\mu_3 z_1 z_2} z_3 \Bigr),
\end{align*}
whence \eqref{assump_lambda} implies 
\begin{align*}
 \imagpart\jb{F(z), Az}_{\C^3} 
 = \sum_{j=1}^{3}\kappa_j \imagpart\lambda_j |z_j|^3 \le 0
\end{align*}
for all $z \in \C^3$. 
In particular,  if \eqref{assump_lambda_strict} is satisfied, 
we can take positive constants $C^*$ and $C_*$ such that 
\begin{align}
-C^* \nu_A(z)^3\le \imagpart\jb{F(z), Az}_{\C^3}
 \le 
-C_* \nu_A(z)^3
\label{strict_dissip}
\end{align}
for $z \in \C^3$, where $\nu_A(z)=\sqrt{\jb{z,Az}_{\C^3}}$. 
Note that 
\begin{align*}
  \sqrt{\kappa_*} |z|_{\C^3} 
  \le \nu_A(z) \le  
  \sqrt{\kappa^*} |z|_{\C^3}
\end{align*}
for $z \in \C^3$, where 
$\kappa_*=\min\{\kappa_1,\, \kappa_2,\, \kappa_3 \}$ and 
$\kappa^*=\max\{\kappa_1,\, \kappa_2,\, \kappa_3 \}$. 
\end{itemize}

Next we set 
$\mathcal{U}(t)= \exp(it\Lambda)$, i.e., 
$$
 \mathcal{U}(t)\phi=
 \left( \exp\left(\frac{it}{2m_j}\Delta\right) \phi_j\right)_{j=1,2,3}
$$
for a $\C^3$-valued smooth function $\phi=(\phi_j)_{j=1,2,3}$. 
Also we set 
$$
 \mathcal{G}\phi(\xi)=\left(-im_j \hat{\phi}_j(m_j\xi)\right)_{j=1,2,3},
$$ 
where $\hat{f}$ denotes the Fourier transform of $f$, i.e.,
$$
 \hat{f}(\xi) 
 = \frac{1}{2\pi}\int_{\R^2} e^{-iy\cdot \xi} f(y) dy.
$$

Now we summarize several useful lemmas. Since they are essentially not new 
or rather standard, we will give only an outline of the proof in the Appendix.

\begin{lem} \label{lem_aux1}
Let $s \in [0,2)$. There exists a constant $C$ such that

$$
 \|F(\phi)\|_{H^{s,0}} 
 \le 
 C\|\phi\|_{L^{\infty}} \|\phi\|_{H^{s,0}}.
$$
\end{lem}

\begin{lem} \label{lem_aux2}
Let $\gamma \in [0,1]$ and $s>1+2\gamma$. 
There exists a constant $C$ such that
$$
\left|\| \phi\|_{L^{\infty}}-t^{-1} \|\mathcal{G} \mathcal{U}(-t) \phi \|_{L^{\infty}}\right| 
 \le 
 Ct^{-1-\gamma} \| \mathcal{U}(-t) \phi\|_{H^{0,s}}
$$ 
for $t\geq 1$.
\end{lem}

\begin{lem} \label{lem_aux3}
Let $s \in [0,2)$. Assume that \eqref{assump_m} is satisfied. 
Then there exists a constant $C$ such that
$$
 \|\mathcal{U}(-t) F(\phi)\|_{H^{0,s}} 
 \le 
 C\|\phi\|_{L^{\infty}} \|\mathcal{U}(-t) \phi\|_{H^{0,s}}.
$$
\end{lem}

\begin{lem} \label{lem_aux4}
Let $\gamma \in [0,1/2)$ and $1+2\gamma<s<2$. 
Assume that \eqref{assump_m} is satisfied. 
Then there exists a constant $C$ such that
$$
 \left\|
   \mathcal{G} \mathcal{U}(-t) F(\phi)
   -
   \frac{1}{t} F\bigl(\mathcal{G} \mathcal{U}(-t)\phi\bigr)
 \right\|_{L^{\infty}} 
 \le 
 Ct^{-1-\gamma} \|\mathcal{U}(-t) \phi \|_{H^{0,s}}^2
$$
for $t\geq 1$.
\end{lem}

\section{A priori estimate}\label{sec_apriori}
The argument of this section is almost the same as those of the previous works 
\cite{HLN1}, \cite{L1}, \cite{L2}, \cite{Sh}.  Let $u(t)$ be the solution to 
\eqref{eq}--\eqref{data} in the interval $[0,T]$. We define 
$$
 \|u\|_{X_T}
 =\sup_{t\in [0,T]} \Bigl\{
 (1+t) \|u(t,\cdot)\|_{L^{\infty}} 
 + (1+t)^{-\gamma/3} 
  \bigl(
    \|u(t,\cdot)\|_{H^{s,0}} 
    + 
    \|\mathcal{U}(-t)u(t,\cdot)\|_{H^{0,s}}
  \bigr)
 \Bigr\},
$$
where $1<s<2$ and $0<\gamma <\frac{s-1}2$. We also set 
$\eps=\|\varphi\|_{H^{s,0}}+\|\varphi\|_{H^{0,s}}$.

\begin{lem} \label{lem_apriori}
There exists a constant $C$, independent of $T$, such that
 $$
  \|u\|_{X_T} 
  \le 
 C \eps + C \|u\|_{X_T}^2.
 $$
\end{lem}
\begin{rmk}
The above estimate implies that there exists a constant $C_0 >0$, 
which does not depend on $T$, such that 
\begin{equation}
 \label{apriori_final}
  \|u\|_{X_T} \le C_0\eps
\end{equation}
if we choose $\eps$ sufficiently small. 
Proposition~\ref{prp_basic} is an immediate consequence of this 
{\it a priori} bound and the standard local existence theorem.
\end{rmk}

\begin{proof}[Proof of Lemma~\ref{lem_apriori}]
In what follows, we denote several positive constants 
by the same letter $C$, which may vary from one line to another. 
First we consider the estimates for $\|u(t,\cdot)\|_{H^{s,0}}$ and 
$\|\mathcal{U}(-t)u(t,\cdot)\|_{H^{0,s}}$. 
By the standard energy inequality combined with 
Lemmas~\ref{lem_aux1} and \ref{lem_aux3}, we have  
\begin{align}
 \|u(t,\cdot)\|_{H^{s,0}} 
 \le& 
 \|\varphi\|_{H^{s,0}} 
 +  
 \int_{0}^{t} \|F(u(\tau,\cdot))\|_{H^{s,0}} d\tau
 \nonumber\\
 \le& 
 C\eps 
 + 
 C \int_{0}^{t} 
   \|u(\tau,\cdot)\|_{L^{\infty}}\|u(\tau,\cdot)\|_{H^{s,0}} 
  d\tau
 \nonumber\\
 \le& 
 C\eps + C \|u\|_{X_T}^2 \int_{0}^{t}  \frac{d\tau}{(1+\tau)^{1-\gamma/3}}
 \nonumber\\
 \le&
 C \eps + C \|u\|_{X_T}^2 (1+t)^{\gamma/3}
 \label{est_energy1}
\end{align}
and
\begin{align}
 \|\mathcal{U}(-t)u(t,\cdot)\|_{H^{0,s}} 
 \le& 
  \|\varphi\|_{H^{0,s}} 
  +  
  \int_{0}^{t} \|\mathcal{U}(-\tau)F(u(\tau,\cdot))\|_{H^{0,s}} d\tau
 \nonumber\\
 \le& 
  C\eps 
  + 
  C \int_{0}^{t} 
     \|u(\tau,\cdot)\|_{L^{\infty}}
     \|\mathcal{U}(-\tau)u(\tau,\cdot)\|_{H^{0,s}} 
 d\tau
 \nonumber\\
 \le& 
  C\eps + C \|u\|_{X_T}^2 \int_{0}^{t}  \frac{d\tau}{(1+\tau)^{1-\gamma/3}}
 \nonumber\\
 \le&
  C \eps + C \|u\|_{X_T}^2 (1+t)^{\gamma/3},
 \label{est_energy2}
\end{align}
respectively. Next we consider the $L^{\infty}$-bound for $u(t)$. 
In the case of $t\le 1$, the standard Sobolev embedding and 
\eqref{est_energy1} lead to
\begin{align}
 (1+t) \|u(t,\cdot)\|_{L^{\infty}} 
 \le 
 2\|u(t,\cdot)\|_{L^{\infty}}
 \le 
 C \|u(t,\cdot)\|_{H^{s,0}}
 \le  
 C(\eps + \|u\|_{X_T}^2).
 \label{est_pointwise1}
\end{align}
From now on, we focus on $t\ge 1$. 
We set 
$$
 \alpha(t, \xi)=\mathcal{G}\bigl( \mathcal{U}(-t)u(t,\cdot)\bigr)(\xi)
$$ 
and 
$$
 r(t, \xi)
 =
 \mathcal{G} \bigl(\mathcal{U}(-t) F(u(t,\cdot))\bigr)(\xi)
 -
 \frac{1}{t} F(\alpha(t, \xi))
$$
so that 
\begin{align*}
 i\pa_t \alpha(t,\xi) 
 =& 
 \mathcal{G} \bigl(\mathcal{U}(-t) (i\pa_t +\Lambda)u(t, \cdot)\bigr)(\xi)\\
 =&  
 \mathcal{G}\bigl(\mathcal{U}(-t) F(u(t,\cdot))\bigr)(\xi)\\
 =&
 \frac{1}{t}F(\alpha(t,\xi)) 
 +
 r(t, \xi).
\end{align*}
By the Sobolev inequality and \eqref{est_energy2}, we get
\begin{align*}
 \left\|\alpha(1,\cdot) \right\|_{L^\infty}
 \le  
  C \left\| {\mathcal U}(-1)u(1,\cdot)\right\|_{H^{0,s}}
 \le  
  C\eps+C\|u\|_{X_T}^2.
\end{align*}
By Lemma~\ref{lem_aux4}, we have 
\begin{align}
 \|r(t, \cdot)\|_{L^{\infty}}
 &\le
 C t^{-1-\gamma} \|\mathcal{U}(-t) u(t)\|_{H^{0,s}}^2 \nonumber\\
 &\le
 C t^{-1-\gamma} \Bigl\{(1+t)^{\frac{\gamma}{3}} \|u\|_{X_T} \Bigr\}^2 \nonumber\\
 &\le
 C t^{-1-\frac{\gamma}{3}} \|u\|_{X_T}^2.
 \label{Est_r}
\end{align}
Then it follows from the straightforward calculation that 
\begin{align}
 \pa_t \Bigl(\nu_A(\alpha(t,\xi))^2 \Bigr) 
 &=
 2\imagpart \JB{i \pa_t \alpha(t,\xi), A\alpha(t,\xi)}_{\C^3}
 \nonumber\\
 &=
 \frac{2}{t}\imagpart \jb{F(\alpha(t,\xi)), A\alpha(t,\xi)}_{\C^3} 
  + 
 2\imagpart \jb{r(t,\xi),  A\alpha(t,\xi)}_{\C^3}
 \nonumber\\
 &\le  
 0+ 2 \nu_A(r(t,\xi)) \nu_A(\alpha (t,\xi)), 
 \label{EqNuA}
\end{align}
which leads  to 
\begin{equation*}
 \pa_t \sqrt{\delta+\nu_A(\alpha(t,\xi))^2} 
 \le \nu_A(r(t,\xi)) 
 \le C\|r(t,\cdot)\|_{L^{\infty}} 
 \le C \|u\|_{X_T}^2 t^{-1-\frac{\gamma}{3}} 
\end{equation*} 
for any $\delta>0$.
Integrating with respect to $t$, and letting $\delta\to +0$, we obtain 
\begin{align*}
\nu_A(\alpha(t,\xi)) 
 \le  
 C\|\alpha(1,\cdot)\|_{L^{\infty}} 
 + 
 C \|u\|_{X_T}^2 \int_{1}^{t} \tau^{-1-\frac{\gamma}{3}} d \tau, 
\end{align*}
whence
\begin{equation}
 \|\alpha(t,\cdot)\|_{L^{\infty}} 
 \le 
 C\eps + C \|u\|_{X_T}^2.
\label{Est_Alpha}
\end{equation}
From Lemma~\ref{lem_aux2} and \eqref{est_energy2}, we deduce that
\begin{align}
 \|u(t,\cdot)\|_{L^{\infty}} 
 &\le 
 t^{-1}\|\alpha(t,\cdot)\|_{L^{\infty}}  
 + C t^{-1-2\gamma/3} (t^{-\gamma/3}\|\mathcal{U}(-t)u(t,\cdot)\|_{H^{0,s}})
 \nonumber\\
 &\le  
 Ct^{-1}(\eps +  \|u\|_{X_T}^2)
 \label{est_pointwise2}
\end{align}
for $t\ge 1$. 

By \eqref{est_energy1}, \eqref{est_energy2}, \eqref{est_pointwise1} and 
\eqref{est_pointwise2}, we obtain the desired estimate.
\end{proof}

\section{Proof of Theorem~\ref{thm_main}}
Now we are in a position to finish the proof of Theorem~\ref{thm_main}. 
It is enough to show 
\begin{align}
 \sup_{\xi \in \R^2}\nu_A(\alpha(t,\xi))
 \le 
  \frac{C}{\log t} 
 \label{main_est}
\end{align}
for $t\ge 2$, because we already know that 
\begin{align*}
 (1+t)\|u(t,\cdot)\|_{L^{\infty}} 
 \le 
 C\sup_{\xi \in \R^2}\nu_A(\alpha(t,\xi))  
 + C \eps t^{-2\gamma/3}
\end{align*}
by virtue of Lemma~\ref{lem_aux2}, \eqref{apriori_final} and 
\eqref{est_energy2}. 

To prove \eqref{main_est}, we put $\Phi(t):=\nu_A(\alpha(t,\xi))^2$ 
(with $\xi \in \R^2$ being regarded as a parameter) 
and compute 
\begin{align*}
 \frac{d}{dt} \Bigl( (\log t)^{{3}} {\Phi(t)} \Bigr)
 =
 (\log t)^{3} \frac{d \Phi}{dt}(t)
 +  
 \frac{3(\log t)^{2}}{t} \Phi(t).
\end{align*}
Similarly to \eqref{EqNuA}, 
it follows from \eqref{strict_dissip}, \eqref{Est_r}, \eqref{Est_Alpha} 
and \eqref{apriori_final} that
\begin{align*}
 \frac{d \Phi}{dt}(t)
 &=
 \frac{2}{t}\imagpart \jb{F(\alpha(t,\xi)), A\alpha(t,\xi)}_{\C^3} 
  + 
 2\imagpart \jb{r(t,\xi),  A\alpha(t,\xi)}_{\C^3}
 \nonumber\\
 &\le 
 -\frac{2}{t} C_*\, \Phi(t)^{3/2} + \frac{C\eps^3}{t^{1+\gamma/3}},
\end{align*}
where $C_*$ is the constant appearing in \eqref{strict_dissip}.
We also have
\begin{align*}
 \Phi(t)
 =
 \left(\frac{1}{C_*^2 (\log t)^{2}} \right)^{1/3} 
 \Bigl\{ C_*(\log t) \Phi(t)^{3/2} \Bigr\}^{2/3}
 \le 
 \frac{1}{3C_*^2 (\log t)^{2}} + \frac{2 C_*}{3}\, (\log t) \Phi(t)^{3/2}
\end{align*}
by the Young inequality. 
%
Piecing them together, we obtain
\begin{align*}
 \frac{d}{dt} \Bigl( (\log t)^{3} \Phi(t) \Bigr)
 \le 
 \frac{C}{t} 
 + \frac{C \eps^{3} (\log t)^3 }{t^{1+\gamma/3}}.
\end{align*}
Integrating with respect to $t$, we arrive at 
$$
 (\log t)^{3} \Phi(t)
 \le 
  C \eps^{2}
  + 
  \int_{2}^{t} \biggl(
   \frac{C}{\tau } + \frac{C\eps^{3} (\log \tau)^3}{\tau^{1+\gamma/3}} 
  \biggr) d\tau
 \le 
 C \log t,
$$
whence $\Phi(t)\le C/(\log t)^2$ for $t\ge 2$, as required.
\qed\\


Finally, we discuss the optimality of the decay rate $O((t\log t)^{-1})$. 
For simplicity, let $\varphi(x)=\delta \psi(x)$ with 
$\psi(\not\equiv 0)\in H^{s,0}\cap H^{0,s}$ and $\delta>0$ 
(note that $\eps=\|\varphi\|_{H^{s,0}}+\|\varphi\|_{H^{0,s}}\le C\delta$).
Then 
we can also show that the solution does not decay strictly faster than 
$t^{-1}(\log t)^{-1}$ as $t \to \infty$ 
if $\delta$ is small enough. 
Indeed, suppose that
\begin{align*}
 \lim_{t\to \infty} t(\log t)\|u(t,\cdot)\|_{L^\infty}=0
\end{align*}
holds true. 
Then, it follows from Lemma~\ref{lem_aux2}, \eqref{apriori_final} and 
\eqref{est_energy2} that
\begin{align}
(\log t)\Phi(t)^{1/2}
&\le 
 C t(\log t) t^{-1}\|\alpha(t,\cdot)\|_{L^\infty}
\nonumber\\
&\le 
 C t(\log t)\left(\|u(t,\cdot)\|_{L^\infty}+C\delta t^{-1-2\gamma/3}\right)
\nonumber\\
&\to 0
 \label{est_contra}
\end{align}
as $t\to \infty$. Hence, if $\delta$ is sufficiently small, we have
$$
C^*(\log t) \Phi(t)^{1/2}\le 1
$$
for $t\ge {2}$, where $C^*$ is the constant appearing in 
\eqref{strict_dissip}.
Similarly to the proof of Theorem~\ref{thm_main},
we have 
\begin{align*}
\frac{d}{dt}\left((\log t)^2\Phi(t)\right)\ge & 
\frac{2(\log t)}{t}\Phi(t)\bigl(1-C^*(\log t)\Phi(t)^{1/2}\bigr)-\frac{C(\log t)^2\delta^3}{t^{1+\gamma/3}}\ge -\frac{C(\log t)^2\delta^3}{t^{1+\gamma/3}},
\end{align*}
which yields
$$
(\log t)^2\Phi(t)
\ge 
 (\log {2})^2\Phi({2})
 -
 \int_{{2}}^t\frac{C(\log \tau)^2\eps^3}{\tau^{1+\gamma/3}}d\tau
 \ge 
 C\delta^2-C'\delta^3>0
$$
for small $\delta$ with some positive constants $C$ and $C'$. This contradicts 
\eqref{est_contra}.


\appendix \section{Appendix}
For the convenience of the readers, we give an outline of the proof of 
the four lemmas stated in Section~\ref{sec_aux}.

\begin{proof}[Proof of Lemma~\ref{lem_aux1}.]
For $s \in [0,2)$, we have
\begin{align}
 \bigl\|\, (-\Delta)^{s/2}(|f|f) \,\bigr\|_{L^2} 
 \le C\|f\|_{L^{\infty}} \|(-\Delta)^{s/2}f\|_{L^2}
 \label{est_aux1}
\end{align}
and
\begin{align}
 \|(-\Delta)^{s/2}(f g)\|_{L^2} 
 \le 
 C(\|f\|_{L^{\infty}} \|(-\Delta)^{s/2}g\|_{L^2} 
 +  
 \|(-\Delta)^{s/2}f\|_{L^2} \|g\|_{L^{\infty}})
 \label{est_aux2}
\end{align}
(see, e.g., \cite{GOV} and \cite{KP} for the proof). 
The desired estimate follows from them immediately.
\end{proof}

\begin{proof}[Proof of Lemma~\ref{lem_aux2}.]
By a simple calculation, we can see that $\mathcal{U}(t)$ is decomposed into 
the following forms:
\begin{align}
\mathcal{U}(t) 
 =\mathcal{M}(t) \mathcal{D}(t) \mathcal{G} \mathcal{M}(t)
 =\mathcal{M}(t) \mathcal{D}(t) \mathcal{W}(t) \mathcal{G},
 \quad t\ne 0,
 \label{mdfm}
\end{align}
where 
$\mathcal{M}(t)=\mathcal{E}_{\theta}$ with $\theta=|x|^2/(2t)$, 
$\bigl(\mathcal{D}(t) \phi\bigr) (x)=\frac{1}{t}\phi(\frac{x}{t})$, and 
$\mathcal{W}(t)=\mathcal{G} \mathcal{M}(t) \mathcal{G}^{-1}$. 
Note that 
$|e^{i\theta}-1|\le C |\theta|^{\gamma}$ for $\gamma \in [0,1]$ and 
$(1+|x|)^{2\gamma-s} \in L^{2}(\R^2)$ if $s> 1+2\gamma$.  
They imply 
\begin{align}
 \left\|{\mathcal G}({\mathcal M}(t)-1)\psi\right\|_{L^{\infty}}
 &\le 
 C \left\|(\mathcal{M}(t)-1)\psi\right\|_{L^1}
 \nonumber\\
 &\le 
 C t^{-\gamma} \left\| |x|^{2\gamma}\psi \right\|_{L^1}
 \nonumber\\
 &\le 
 C t^{-\gamma} 
 \left\|(1+|x|)^{2\gamma-s}\right\|_{L^2} 
 \left\|(1+|x|)^s \psi \right\|_{L^2}
 \nonumber\\ 
 &\le 
 C t^{-\gamma} \left\|\psi\right\|_{H^{0,s}}. 
 \label{est_w}
\end{align}
Since we have 
\begin{align*}
 \|\phi\|_{L^{\infty}}
 &= 
 \|
  \mathcal{M}(t)\mathcal{D}(t){\mathcal G}\mathcal{M}(t)\mathcal{U}(-t) \phi
 \|_{L^{\infty}}\\
 &=
 t^{-1}\|\mathcal{G}\mathcal{M}(t) \mathcal{U}(-t) \phi\|_{L^{\infty}},
\end{align*}
it follows from \eqref{est_w} that
\begin{align*}
\left|\|\phi\|_{L^\infty}-t^{-1}\|{\mathcal G}{\mathcal U}(-t)\phi\|_{L^\infty}\right| &\le
 t^{-1}
 \|\mathcal{G}(\mathcal{M}(t)-1)\mathcal{U}(-t) \phi\|_{L^{\infty}} 
 \\
 &\le
 Ct^{-1-\gamma}\|\mathcal{U}(-t) \phi\|_{H^{0,s}}.
\end{align*}
\end{proof}

\begin{proof}[Proof of Lemma~\ref{lem_aux3}.]
By \eqref{assump_m}, or equivalently \eqref{gauge_inv}, we have
\begin{align}
 \mathcal{M}(-t)F(\phi)=F(\mathcal{M}(-t)\phi).
 \label{gauge_inv2}
\end{align}
We also note that
$$
  \mathcal{U}(t) |x|^s \mathcal{U}(-t) 
  =\mathcal{M}(t) \begin{pmatrix} 
    \frac{t^s (-\Delta)^{s/2}}{m_1^s} & 0 &0\\
    0 & \frac{t^s (-\Delta)^{s/2}}{m_2^s} & 0 \\
   0 & 0 & \frac{t^s (-\Delta)^{s/2}}{m_3^s} 
   \end{pmatrix}\mathcal{M}(-t). 
$$
By using \eqref{est_aux1}, \eqref{est_aux2} and the above identities, 
we deduce that 
\begin{align*}
 \left\| |x|^s\mathcal{U}(-t) F(\phi) \right\|_{L^{2}}
 &\le
 Ct^{s} \left\| (-\Delta)^{s/2}F(\mathcal{M}(-t)\phi) \right\|_{L^2}\\
 &\le
 C t^s \left\| \mathcal{M}(-t)\phi \right\|_{L^{\infty}} 
       \left\| (-\Delta)^{s/2} (\mathcal{M}(-t)\phi) \right\|_{L^2}\\
 &\le
  C \|\phi\|_{L^{\infty}} \left\| |x|^s \mathcal{U}(-t) \phi \right\|_{L^2}.
\end{align*}
\end{proof}

\begin{proof}[Proof of Lemma~\ref{lem_aux4}.]
We put $\beta=\mathcal{G}\mathcal{U}(-t)\phi$.
By \eqref{mdfm} and \eqref{gauge_inv2}, we have
\begin{align*}
 \mathcal{G}\mathcal{U}(-t)F(\phi)
 &=
 \mathcal{W}(t)^{-1}\mathcal{D}(t)^{-1}\mathcal{M}(t)^{-1}
 F(\mathcal{M}(t)\mathcal{D}(t)\mathcal{W}(t)\beta)\\
 &=
 \mathcal{W}(t)^{-1}\mathcal{D}(t)^{-1}
 F(\mathcal{D}(t)\mathcal{W}(t)\beta)\\
 &=
 \frac{1}{t} \mathcal{W}(t)^{-1}F(\mathcal{W}(t)\beta)\\
 &=
 \frac{1}{t}F(\mathcal{G}\mathcal{U}(-t)\phi)+R(t),
\end{align*}
where
\begin{align*}
 &R(t)=\frac{1}{t}(R_1(t)+R_2(t)),\\
 &R_1(t)=(1-\mathcal{W}(t))\mathcal{W}(t)^{-1}F(\mathcal{W}(t)\beta),\\
 &R_2(t)=F(\mathcal{W}(t)\beta) - F(\beta).
\end{align*}
\eqref{est_w} yields
\begin{align*}
\left\|(\mathcal{W}(t)-1)\psi\right\|_{L^\infty}\le & Ct^{-\gamma}\left\|\mathcal{G}^{-1}\psi\right\|_{H^{0,s}}\le Ct^{-\gamma}\|\psi\|_{H^{s,0}}.
\end{align*}
Hence, by Lemma~\ref{lem_aux1} and the Sobolev embedding, we have
\begin{align*}
 \|R_1(t)\|_{L^{\infty}}
 & \le
 Ct^{-\gamma} \left\|\mathcal{W}(t)^{-1}F(\mathcal{W}(t)\beta)\right\|_{H^{s,0}}\\
 &\le Ct^{-\gamma} \|F(\mathcal{W}(t)\beta)\|_{H^{s,0}}\\
 &\le 
 Ct^{-\gamma} \|\mathcal{W}(t)\beta\|_{H^{s,0}}^2\\
 &\le 
 Ct^{-\gamma} \|\beta\|_{H^{s,0}}^2
\end{align*}
and
\begin{align*}
 \|R_2(t)\|_{L^{\infty}}
 &\le 
 C(\|\mathcal{W}(t)\beta\|_{L^{\infty}} + \|\beta\|_{L^{\infty}})
 \|(\mathcal{W}(t)-1)\beta\|_{L^{\infty}}\\
 &\le
 C (\|\mathcal{W}(t)\beta\|_{H^{s,0}} + \|\beta\|_{H^{s,0}})
 \cdot Ct^{-\gamma}\|\beta\|_{H^{s,0}}\\
 &\le
 Ct^{-\gamma} \|\beta\|_{H^{s,0}}^2.
\end{align*}
Summing up, we arrive at
$$
 \|R(t)\|_{L^{\infty}}
 \le 
 Ct^{-1-\gamma} \|\beta\|_{H^{s,0}}^2
 =
 Ct^{-1-\gamma} \|\mathcal{U}(-t) \phi\|_{H^{0,s}}^2,
$$
as required.
\end{proof}
\medskip
\subsection*{Acknowledgments}

Two of the authors (S.~K. and H.~S.) would like to express their 
gratitude for warm hospitality of Department of Mathematics, 
Yanbian University. 
This work started during their visit there. 
The authors also thank Professor Akitaka Matsumura and 
Professor Hisashi Nishiyama for valuable comments 
on the earlier version of this work.

The work of S.K. is partially supported by JSPS, 
Grant-in-Aid for Scientific Research (C) 23540241. 
The work of H.S. is partially supported by JSPS, 
Grant-in-Aid for Young Scientists (B) 22740089 
and 
Grant-in-Aid for Scientific Research (C) 25400161. 


\end{document}